\documentclass[9.5pt,reqno]{amsart}
\usepackage{amssymb,mathrsfs,graphicx}
\usepackage{epsfig,subfigure}
\usepackage{ifthen}
\usepackage{refcheck}
\usepackage{float}
\usepackage{colortbl}

\usepackage{ifthen} 
\provideboolean{shownotes} 
\setboolean{shownotes}{true} 
\newcommand{\margnote}[1]{
\ifthenelse{\boolean{shownotes}}%
{\marginpar{\raggedright\tiny\texttt{#1}}}%
{}%
}

\newcommand{\hole}[1]{
\ifthenelse{\boolean{shownotes}}%
{\begin{center} \fbox{ \rule {.25cm}{0cm} \rule[-.1cm]{0cm}{.4cm}
\parbox{.85\textwidth}{\begin{center} \texttt{#1}\end{center}} \rule
{.25cm}{0cm}}\end{center}} {} }

\title[A gradient flow approach of propagation of chaos]{A gradient flow approach of uniform in time propagation of chaos for particles in double a well confinement}

\author[Salem]{Samir Salem}
\address[Samir Salem]{\newline CEREMADE UMR 7534, \newline
    Universit\'e Paris Dauphine, Place du Mar\'echal de Tassigny, Cedex Paris, France}
\email{salem@ceremade.dauphine.fr}

\numberwithin{equation}{section}

\newtheorem{theorem}{Theorem}[section]
\newtheorem{lemma}{Lemma}[section]
\newtheorem{corollary}{Corollary}[section]
\newtheorem{proposition}{Proposition}[section]
\newtheorem{remark}{Remark}[section]
\newtheorem{definition}{Definition}[section]

\newcommand{\R}{\mathbb R}

\newcommand{\B}{\mathbb{B}}

\newcommand{\II}{\mathcal{I}}

\newcommand{\MM}{\mathcal{M}}

\newcommand{\XX}{\mathcal{X}}

\newcommand{\PP}{\mathcal{P}}
\newcommand{\RR}{\mathcal{R}}

\newcommand{\FF}{\mathcal{F}}
\newcommand{\CC}{\mathcal{C}}

\newcommand{\JJ}{\mathcal{J}}

\newcommand{\mei}{\frac{1}{N}\sum_{i=1}^N}
\newcommand{\mej}{\frac{1}{N}\sum_{j=1}^N}

\newcommand{\e}{\varepsilon}
\newcommand{\mt}{\mathcal{T}}
\newcommand{\lal}{\langle}
\newcommand{\ral}{\rangle}

\newcommand{\lt}{\left}
\newcommand{\rt}{\right}
\newcommand{\pa}{\partial}
 
\newcommand{\mb}{\mathbf{1}}

\newcommand{\bq}{\begin{equation}}
\newcommand{\eq}{\end{equation}}

\newcommand{\LL}{\mathcal{L}}

\newcommand{\E}{\mathbb{E}}
\def\charf {\mbox{{\text 1}\kern-.30em {\text l}}}

\begin{document}
\allowdisplaybreaks

\date{\today}



\begin{abstract} 
We provide an estimation of the dissipation of the Wasserstein 2 distance between the law of some interacting $N$-particle system, and the $N$ times tensorized product of solution to the corresponding limit nonlinear conservation law. It then enables to recover classical propagation of chaos results \cite{Szn} in the case of Lipschitz coefficients, uniform in time propagation of chaos in \cite{Mal} in the case of strictly convex coefficients. But also some recent results \cite{Gcou} as the case of particle in a double well potential.
\end{abstract}

\maketitle \centerline{\date}

\tableofcontents

%
%
%
%
\section{Introduction}\label{intro}

This paper introduces a new method to prove the so called propagation of chaos (we refer to the classical lecture notes \cite{Szn}) for the interacting particle system
\bq
\label{eq:PS}
X_t^i=X_0^i+\int_0^t \mej b(X_s^i,X_s^j)ds +\sqrt{2}B_t^i,
\eq
where $((B_t^i)_{t\geq 0})_{i=1\cdots,N}$ are $N$ independent $d$-dimensional Brownian motion, $(X_0^i)_{i=1,\cdots,N}$ are random variables of symmetric joint law $G_0^N\in \PP^{\textit{{sym}}}(\R^{dN})$ independent of the $(B^i)_{i=1\cdots,N}$, and $b:\R^{2d}\mapsto \R^d$ is an interaction field.\newline
For the sake of completeness we recall some basic notions on this topic, and refer to \cite{Szn} for some further explanations. We begin with the
\begin{definition}[Definition 2.1 in \cite{Szn}]
	\label{def:chaos}
	Let $(G_N)_{N\geq 1}$ be a sequence of symmetric probabilities on $E^N$ ($G_N\in \PP_{\mbox{sym} }(E^N) $), with $E$ some polish space. We say that $G_N$ is $g$-chaotic, with $g\in \PP(E)$, if for any $k\geq 2$ and $\phi_1,\cdots,\phi_k\in C_b(E)$ it holds
	\[
	\lim_{N\rightarrow +\infty}\int_{E^N}\bigotimes_{j=1}^k\phi_j G_N=\prod_{j=1}^k\lt(\int_{E}\phi_j g\rt).
	\]
\end{definition}
In other words a sequence $(G_N)_{N\geq 1}$ is $g$-chaotic if and only if for any $k\geq 2$
\[
G_N^k\underset{N\rightarrow+\infty}{\overset{*}{\rightharpoonup}}g^{\otimes k} \quad \mbox{in} \quad \PP(E^k),
\] 
where $G_N^k$ is the $k$-particle marginal of $G_N$. It is well known since the seminal work of McKean (\cite{Mck}) (when $b$ is regular), that if the initial law of particle system \eqref{eq:PS} is $\mu_0$-chaotic, then for any further time $t>0$ the law of the solution at time $t>0$ to the particle system is $\mu_t$-chaotic, where $\mu_t$ is the solution to the nonlinear conservation law
\bq
\label{eq:NLPDE}
\begin{aligned}
\pa_t \mu_t&=\nabla \cdot \lt( \nabla \mu_t -(b*\mu_t)\mu_t , \rt)\\
\end{aligned}
\eq
starting from $\mu_0$ (with the notation $b*\mu(x)=\int_{\R^d}b(x,z)\mu(dz)$). Usually this result is proved by coupling method, which consists in introducing a system of nonlinear (independent, non interacting) particle
\[
Y_t^i=X_0^i+\int_0^t b*\mu_s(Y_s^i)ds +\sqrt{2}B_t^i,
\]
driven by the same independent Brownian motions as in \eqref{eq:PS} and the nonlinear force of the limiting solutions. This technique is then called coupling, since the only random objects used to define the trajectories $((X_t^i)_{t\geq 0})_{i=1\cdots,N}$ and $((Y_t^i)_{t\geq 0})_{i=1\cdots,N}$ are the same, namely the initial condition  $(X_0^i)_{i=1\cdots,N}$ and the independent Brownian motions $((B_t^i)_{t\geq 0})_{i=1\cdots,N}$. Therefore the trajectories are built on the same stochastic space, and one says that they are coupled.

Assuming that the $(X_0^i)_{i=1,\cdots,N}$ are i.i.d.of law $\mu_0$ (which is stronger than $\mu_0$-chaotic), makes the $(Y_t^i)_{i=1,\cdots,N}$ i.i.d.of law $\mu_t$ for any $t>0$. Therefore (assuming some second order moments for the sake of simplicity) it holds
\begin{align*}
\E\lt[W^2_2\lt( \mei \delta_{X_t^i},\mu_t  \rt)\rt]&\leq 2\E\lt[W^2_2\lt( \mei \delta_{X_t^i}, \mei \delta_{Y_t^i} \rt)\rt]+2\E\lt[W^2_2\lt( \mei \delta_{Y_t^i},\mu_t  \rt)\rt]\\
&\leq  2\E \lt[ \mei |X_t^i-Y_t^i|^2\rt]+2\E\lt[W^2_2\lt( \mei \delta_{Y_t^i},\mu_t  \rt)\rt].
\end{align*}
But it is straightforward to obtain 
\begin{align*}
\mei\frac{|X_t^i-Y_t^i|^2}{2}&=\mei\int_0^t\lal \mej b(X_s^i,X_s^j)-\int b(Y_s^i,z)\mu_s(dz) ,X_s^i-Y_s^i\ral\\
&\leq \mei\int_0^t\lal \mej b(X_s^i,X_s^j)-\mej b(Y_s^i,Y_s^j) ,X_s^i-Y_s^i\ral ds\\
&+\mei\int_0^t\lal\mej b(Y_s^i,Y_s^j)-\int b(Y_s^i,z)\mu_s(dz) ,X_s^i-Y_s^i\ral ds\\
&\leq \|b\|_{Lip}\int_0^t \mei\frac{|X_s^i-Y_s^i|^2}{2} ds+\mei\int_0^t \lt(\mej b(Y_s^i,Y_s^j)-\int b(Y_s^i,z)\mu_s(dz)\rt)^2ds.
\end{align*}
Then since the $(Y_t^i)_{i=1,\cdots,N}$ are i.i.d.of law $\mu_t$ we straightforwardly obtain 
\begin{align*}
\E&\lt[ \lt(\mej b(Y_s^i,Y_s^j)-\int b(Y_s^i,z)\mu_s(dz)\rt)^2  \rt]\\
&\quad =\frac{1}{N^2}\sum_{j,k=1}^N\E\lt[ \lt( b(Y_s^i,Y_s^j)-\int b(Y_s^i,z)\mu_s(dz)\rt)\lt( b(Y_s^i,Y_s^k)-\int b(Y_s^i,z)\mu_s(dz)\rt)  \rt]\\
&\quad =\frac{1}{N^2}\sum_{j=1}^N\E\lt[ \lt( b(Y_s^i,Y_s^j)-\int b(Y_s^i,z)\mu_s(dz)\rt)^2  \rt]\leq \frac{4\|b\|^2_{L^\infty}}{N},
\end{align*}
so that Gronwall's inequality yields 
\begin{align*}
	\E \lt[ \mei |X_t^i-Y_t^i|^2\rt]\leq \frac{4\|b\|^2_{L^\infty}}{N}e^{\|b\|_{Lip}t}.
\end{align*}
Using again the fact that  the $(Y_t^i)_{i=1,\cdots,N}$ are i.i.d.of law $\mu_t$, then \cite{FG} we obtain for some constant $C>0$ depending on $(\mu_t)_{t\geq 0}$
\begin{align*}
\E\lt[W^2_2\lt( \mei \delta_{X_t^i},\mu_t  \rt)\rt]\leq \frac{4\|b\|^2_{L^\infty}e^{\|b\|_{Lip}t}+C}{N}.
\end{align*}

Note that the only feature of the $(Y_t^i)_{i=1,\cdots,N}$ which is used by this proof is that they are i.i.d.of law $\mu_t$ for any $t>0$, and are built on the same probability space as the $(X_t^i)_{i=1,\cdots,N}$ (coupled). Due to the choice which has been made to drive the trajectories of the $(Y_t^i)_{i=1,\cdots,N}$ by the exact same Brownian motions as the ones which drive the $(X_t^i)_{i=1,\cdots,N}$, this kind of coupling is called synchronous. In some sense, it is the cheapest coupling one can think of. Indeed it does not take in account the diffusion, and makes the proof of propagation of chaos for the particle system \eqref{eq:PS} very similar to the proof of the mean field limit for the same particle system without diffusion (see for instance \cite{Dob}). This is the cause of some confusion between these two notions in the literature.   \newline
Later another coupling approach has been introduced, which turns out to be particularly powerful in the context of uniform in time propagation of chaos. Instead of defining the nonlinear particles as above, define them as
\begin{align*}
&Y_t^i=X_0^i+\int_0^t b*\mu_s(Y_s^i)ds +\sqrt{2}\int_0^t\lt( I_d-2\mathbf{e}^i_s \lal\mathbf{e}^i_s,\cdot\ral \rt)dB_s^i, \ t\leq T_i, \ Y_t^i=X_t^i \ t\geq T_i,\\
\end{align*}
with $T_i=\inf\{t\geq 0, X_t^î=Y_t^i\}$ and $\mathbf{e}^i_t=\frac{X_t^i-Y_t^i}{|X_t^i-Y_t^i|}$. Defined this way, the $(Y_t^i)_{i=1,\cdots,N}$ still enjoy the properties used in the above argument. But the advantage is that now, compared to the synchronous coupling approach, the difference between $X_t^i$ and $Y_t^i$ includes the diffusive part $-2\mb_{t\leq T_i}\int_0^t\mathbf{e}^i_s \lal\mathbf{e}^i_s,dB_s^i\ral$ which can be particularly handy in the context of asymptotic in time estimate. This technique has been used recently in a general framework \cite{Gcou}, which enables to treat the tricky particular case of particles in a double well confinement interacting through small Lipschitz interaction.  \newline
The aim of this paper is to address this question in an analytical framework. The rest of the paper is organized as follows. In Section 2, we recall some earlier results about gradient flow in probability spaces and dissipation of Wasserstein 2 metric, give a new proof of the classical results presented in the Introduction and state the main theorem of this paper. In Section 3, we give some tools about optimal transport on $\PP^{\textit{sym}}(\R^{dN})$ which enables to use the techniques recalled in Section 2 at the microscopic level. In section 4 we give the proof of the main theorem of this paper.\newline
 In the rest of the paper (and sometimes in the above Introduction as well), $\MM_k(\R)$ stands for the square matrix of size $k$ with real coefficient, $Tr\lt[\cdot\rt]$ stands for the trace operator on this space, and $\cdot^\perp $ stands for the transpose (for a vector or rectangular matrix as well). $\PP_k^{\textit{sym}}(\R^{dN})$ will denote, depending on the context, either for the set of symmetric probability measures over $\R^{dN}$, with $N\geq 2$, and with order $k$ moment. When we will consider some probability measure, it will implicitly always be assumed absolutely continuous w.r.t. the Lebesgue with smooth density, and we will often abuse the confusion between a measure and its density. For a convex functional $\psi$, we denote $\psi^*$ its Legendre convex conjugate. $W_2$ stands for the Wasserstein metric (see for instance \cite{Vil1}). We will use the shortcuts $\sum_{i,j=1}^N=\sum_{i=1}^N\sum_{j=1}^N$ and $\sum_{i\neq j}^N=\sum_{i=1}^N\sum_{j\neq i}^N$. For a function $f\in \CC^2(\R^{dN})$, and $i=1\cdots,N$ we denote $\nabla_i f \in \R^d$ the gradient, and $\nabla^2_{i,i}f \in \MM_d(\R)$ the Hessian matrix w.r.t. to the $i$-th component.

%
%

\section{Preliminaries and main results}
The main motivation of this paper is the following. Instead of trying to find a more or less optimal coupling between the law of the particle system \eqref{eq:PS} and the solution to \eqref{eq:NLPDE}, at the level of particles trajectories as described in the Introduction, we can use the gradient flow approach which links optimal transport, entropy and heat flow (watch for instance \cite{Vil}).

\subsection{Wasserstein 2 metric dissipation and WJ inequality}
We begin this section by the following observation. \newline
If we denote $G_t^N\in \PP^{\textit{sym}}(\R^{dN})$ the law of the random vector $(X_t^1,\cdots,X_t^N)$ solution at time $t>0$ to \eqref{eq:PS} then it solves the following Louiville equation
\bq
\label{eq:Loui}
\begin{aligned}
	&\pa_t G_t^N=\Delta G_t^N-\nabla\cdot\lt( \mathbf{b}^NG_t^N  \rt)\\
	&G^N_{|t=0}=G_0^N,
\end{aligned}
\eq
with
\[
\mathbf{b}^N:(x_1,\cdots,x_N)\in \R^{dN}\mapsto\lt(  \mei b(x_1,x_i),\cdots, \mei b(x_N,x_i) \rt)\in \R^{dN}.
\]
Which we can rewrite
\begin{align*}
	&\pa_t G_t^N=\nabla \cdot \lt( \lt( \nabla \ln G_t^N- \mathbf{b}^N \rt) G_t^N   \rt) \\
	&G^N_{|t=0}=G_0^N
\end{align*} 
so that $G_t^N$ can be seen as the pushed forward of $G_0^N$ by the application $\xi^{N}_{t,0}$ defined as
\begin{align*}
	\xi^{N}_{t,0}(X^N)=X^N+\int_0^t\lt(\mathbf{b}^N-\nabla \ln G_s^N\rt)(\xi^{N}_{s,0}(X^N))ds
\end{align*}

Let now be $(\mu_t)_{t\geq 0}$ the solution to \eqref{eq:NLPDE} starting from $\mu_0$. Similarly $\mu_t^{\otimes N}$ is the pushed forward of $\mu_0^{\otimes N}$ by the application $\zeta_{t,0}^{\otimes N}$
\begin{align*}
	\zeta_{t,0}(x)=x+\int_0^t\lt(b*\mu_s-\nabla \ln \mu_s\rt)(\zeta_{s,0}(x))ds
\end{align*}

Before giving the main result of this section, we recast some notions about dissipation of $W_2$ distance along continuous curves of $\PP_2$, taken from \cite{BGG1,BGG2}

\begin{definition}
	Let be $\mu,\nu\in\PP_2(\R^d)$ and $\psi\in \CC^2(\R^d)$ the maximizing Kantorovitch potential (m.K.p.) associated to the pair $(\nu,\mu)$ as given by Brenier's Theorem, i.e. such that $\nabla\psi \# \nu=\mu$ optimally w.r.t. the $W_2$ metric and define the functional $\JJ$ as
	\label{WJineeq}
	\bq
	\label{eq:WJineeq}
	\begin{aligned}
		\JJ(\mu|b,\nu):&=\int_{\R^d}\lt(\Delta \psi(x)+\Delta \psi^*(\nabla\psi)-2d\rt)\nu(dx)\\
		&+\int_{\R^{2d}} \lal b(\nabla \psi(x),\nabla \psi(y))-b(x,y)   ,\nabla \psi(x)-x  \ral\nu(dx)\nu(dy).
	\end{aligned}
	\eq
	We say that $\nu\in \PP_2(\R^d)$ satisfies a $WJ(\kappa)$ inequality if and only if there is some constant $\kappa>0$ such that for any $\mu\in \PP_2(\R^d)$ it holds
	\[
	\kappa W_2^2(\mu,\nu)\leq \JJ(\mu|b,\nu).
	\]
	More generally, $\nu\in \PP_2(\R^d)$ satisfies a \textit{symmetric} $WJ(\kappa)$ inequality if and only if there is some constant $\kappa>0$ such that for any $N\geq 2$, $G^N\in \PP_2^{\textit{sym}} (\R^{dN})$ it holds
	\label{chWJineeq}
	\bq
	\label{eq:chWJineeq}
	\begin{aligned}
		\kappa W_2^2(G^N,\nu^{\otimes N})\leq \JJ(G^N|\mathbf{b}^N,\nu^{\otimes N})&=\int_{\R^{dN}}\lt(\Delta \psi^N(x)+\Delta \psi^{*N}(\nabla\psi^N)-2dN\rt)\nu^{\otimes N}(dx)\\
		&+\frac{1}{N}\sum_{i,j=1}^N\int_{\R^{dN}} \lal b(\nabla_i \psi^N(x),\nabla_j \psi^N(x))-b(x_i,x_j)   , \nabla_i \psi^N(x)-x_i  \ral\nu^{\otimes N}(dx)
	\end{aligned}
	\eq
	where $\psi^N$ is the m.K.p. associated to the pair $(\nu^{\otimes N},G^N)$. 
\end{definition}

\begin{remark}
	\label{rq:WJ}
	The first part of this Definition is taken from \cite[Definition 3.1]{BGG1}. One can check that if $\nu$ satisfies a \textit{symmetric} WJ($\kappa$) inequality, then it satisfies a WJ($\kappa$) inequality in the sense of \cite[Definition 3.1]{BGG1}. Indeed for any $\mu\in \PP_2(\R^d)$ if $\psi^N$ is the m.K.p. associated to the pair $(\mu^{\otimes N},\nu^{\otimes N})$ and $\psi$ the one associated to the pair $(\mu,\nu)$, one has for any $(x_1,\cdots,x_N)\in \R^{dN}$  
	\[
	\psi^N(x_1,\cdots,x_N)=\sum_{i=1}^N\psi(x_i).
	\]
	So that
		\[
		W_2^2(\mu^{\otimes N},\nu^{\otimes N})=NW_2^2(\mu,\nu), \ \JJ(\mu^{\otimes N}|\mathbf{b}^N,\nu^{\otimes N})=N  \JJ(\mu|b,\nu).
		\]
	Conversely, it does not seem clear whether or not, if $\nu$ satisfies a WJ($\kappa$) inequality in the sense of \cite[Definition 3.1]{BGG1} then it satisfies a \textit{symmetric} WJ($\kappa$) inequality. Even though in practice, the same estimates which enable to establish that some probability measure satisfies a WJ inequality enable to establish that it also satisfies a symmetric WJ inequality, as the reader can check in the proof of Proposition \ref{prop:WJ}.

\end{remark}

The functional $\JJ$ defined so, is the rate of time dissipation of the Wasserstein 2 distance between two solutions to \eqref{eq:NLPDE}. We emphasize that this rate is by construction sharper than the rate of time dissipation of synchronous coupling between two such solutions, which would be given (roughly speaking) by only the second term in the r.h.s. of \eqref{eq:WJineeq} or \eqref{eq:chWJineeq}. In the symmetrical case, this functional appears in the estimate of the time dissipation of the $W_2$ distance between the solution to the Louiville equation associated to the $N$-particles system \eqref{eq:PS} and the $N$ times tensor product of the solution to \eqref{eq:NLPDE}. It is the object of the

\begin{proposition}
	\label{prop:main}
	Let be $N\geq 2$, $(G_t^N)_{t\geq 0}$ be the law of solution to the particle system \eqref{eq:PS} starting with initial condition of law $G_0^N\in \PP^{\textit{sym}}_2(\R^{dN})$,  and $(\mu_t)_{t\geq 0}$ be the solution to \eqref{eq:NLPDE} starting from $\mu_0\in  \PP_2(\R^d)$. Assume that the vector fields $\mathbf{b}^N-\nabla \ln G^N,b*\mu-\nabla \ln \mu$ satisfy 
	\[
	\int_0^t\int_{ \R^{dN} }\lt|\mathbf{b}^N-\nabla \ln G_s^N\rt|^2G_s^N \ ds+\int_0^t\int_{ \R^{d} }\lt|b*\mu_s-\nabla \ln \mu_s\rt|^2\mu_s \ ds<\infty, \ \text{for any} \ t \geq 0,
	\]
	and are locally Lipschitz. Then for any $\eta>0,t\geq 0$ it holds
	\[
	W_2^2(\mu_t^{\otimes N},G^N_t)\leq W_2^2(\mu_0^{\otimes N},G^N_0)-\int_0^t \lt(\JJ(G^N_s|\mathbf{b}^N,\mu^{\otimes N}_s)-\eta W_2^2(\mu_s^{\otimes N},G^N_s)\rt)ds+\eta^{-1}\int_0^t\FF_N(b,\mu_s)  ds,
	\]
	with 
	\[
	\FF_N(b,\mu_s)=\frac{1}{N^2}\sum_{i\neq j}^N \int_{\R^{dN}} \int_{\R^d} \lt| b(x_i,x_j)-b(x_i-z,x_j) \rt|^2\mu(dz) \mu^{\otimes N}(dx).
	\]
\end{proposition}

The proof of this Proposition relies on, by now, classical results, so that we give it in this section 

\begin{proof}
Let be $\psi_t^N$ be the m.K.p. associated to the pair $(\mu_t^{\otimes N},G_t^N)$ (i.e. $\nabla \psi_t^N\#\mu_t^{\otimes N}=G_t^N$ and $\nabla \psi_t^{N*}\#G_t^N=\mu_t^{\otimes N}$). We make a similar use of \cite[Theorem 23.9]{Vil1} or \cite[Theorem 8.4.7]{Amb}, as is done in \cite{BGG1,BGG2} to obtain
	
	\begin{align*}
	W_2^2(G_t^{N},\mu_t^{\otimes N})&=-2\int_0^t \int_{\R^{dN}}\lal \mathbf{b}^N(x)-\nabla \ln G_s^N(x) ,  \nabla \psi_s^{N*}(x) -x\ral G^N_s(dx) \ ds \\
	&- 2\int_0^t \int_{\R^{dN}}\lal \lt(b*\mu_s-\nabla \ln \mu_s\rt)^{\otimes N}(x) ,  \nabla \psi_s^{N}(x) -x\ral \mu_s^{\otimes N}(dx) \ ds.
	\end{align*}
 Then using that $\nabla \psi^{N*}(\nabla \psi^N(x))=x$ we rewrite
	\begin{align*}
	W_2^2(G_t^{N},\mu_t^{\otimes N})&=-2\int_0^t \int_{\R^{dN}}\lal \mathbf{b}^N(\nabla \psi^N_s(x))-\mathbf{b}^N(x) , \nabla \psi_s^N(x)-x \ral\mu_s^{\otimes N}(dx) \ ds \\
	&+2\int_0^t \int_{\R^{dN}}\lal \mathbf{b}^N(x)-b*\mu_s^{\otimes N}(x) , \nabla \psi_s^N(x)-x \ral\mu_s^{\otimes N}(dx) \ ds\\
	&+ 2\int_0^t \lt(\int_{\R^{dN}}\lal \nabla  G_s^N(x) ,  \nabla \psi_s^{N*}(x) -x\ral dx+\int_{\R^{dN}}\lal \nabla \mu_s^{\otimes N}(x) ,  \nabla \psi_s^{N}(x) -x\ral \rt) ds\\
	&=\II_1+\II_2+\II_3,
	\end{align*}
	$\II_1$ is untouched. Then by integration by parts we obtain
		\begin{align*}
		\II_3&=2\int_0^t \lt(\int_{\R^{dN}}\lal \nabla  G_s^N(x) ,  \nabla \psi_s^{N*}(x) -x\ral dx+\int_{\R^{dN}}\lal \nabla \mu_s^{\otimes N}(x) ,  \nabla \psi_s^{N}(x) -x\ral \rt) ds\\
		&=-2\int_0^t\lt(\int_{\R^{dN}}  G_s^N(x) \lt(\Delta \psi^{N*}(x)-dN\rt)  dx + \int_{\R^{dN}}  \mu_s^{\otimes N}(x) \lt(\Delta \psi^{N}(x)-dN\rt)  dx \rt) ds\\
		&=-2\int_0^t\int_{\R^{dN}}\lt(\Delta \psi^{N*}(\nabla \psi(x))+\Delta \psi^{N}(x)-2dN\rt) \mu_s^{\otimes N}(dx)ds.
		\end{align*}

	Finally using Young's inequality yields
	\[
	\II_3\leq 2 \eta \int_0^t W_2^2(G_s^N,\mu_s^{\otimes N})ds+\eta^{-1}\int_0^t\int_{\R^{dN}}\lt|\mathbf{b}^N(x)-(b*\mu_s)^{\otimes N}(x)\rt|^2\mu_s^{\otimes N}(dx).
	\]
	But observe that
	
		\begin{align*}
			\lt| \mathbf{b}^N(x)-(b*\mu)^{\otimes N}(x)\rt|^2&=\sum_{i=1}^N \lt| \mej b(x_i,x_j)-b*\mu(x_i) \rt|^2\\
			&\leq \frac{1}{N^2}\sum_{i=1}^N \sum_{j,k=1}^N \lt(b(x_i,x_j)-b*\mu(x_i) \rt)\lt(b(x_i,x_k)-b*\mu(x_i) \rt),
		\end{align*}
		then observe that for $j\neq k$ we have
		\begin{align*}
			&\int_{\R^{dN}}\lt(b(x_i,x_j)-b*\mu(x_i) \rt)\lt(b(x_i,x_k)-b*\mu(x_i) \rt) \mu^{\otimes N}(dx)\\
			&=\int_{\R^{d(N-2)}} \lt(\int_{\R^d}\lt(b(x_j,x_k)-b*\mu(x_i) \rt) \mu(dx_j)\rt) \lt(\int_{\R^d}\lt(b(x_i,x_k)-b*\mu(x_i) \rt) \mu(dx_k)\rt)\mu^{\otimes (N-2)}(dx'_{j,k})\\
			&=0,
		\end{align*}
		so that
		\begin{align*}
			\FF_N(b,\mu)=&\frac{1}{N^2}\sum_{i=1}^N \sum_{j\neq i}^N\int_{\R^{dN}} \lt(b(x_i,x_j)-b*\mu(x_i) \rt)^2\mu^{\otimes N}(dx).
		\end{align*}
	
\end{proof}

\subsection{Main result}

It is possible to recover from this estimate, some propagation of chaos results obtained by synchronous coupling. All the discussion below relies on the assumption that the vector field $\mathbf{b}^N-\nabla \ln G^N$ matches the required technical assumptions. We will carefully check that these assumptions are indeed fulfilled later in the paper, when $b$ is explicitly defined. \newline
We just obtained in Proposition \ref{prop:main} that
	\begin{align*}
	W_2^2(\mu_t^{\otimes N},G^N_t)&\leq W_2^2(\mu_0^{\otimes N},G^N_0) -2\int_0^t \int_{\R^{dN}} \lt(\Delta \psi_s^N+\Delta \psi_s^{*,N }(\nabla \psi_s^N)-2dN\rt)\mu_s^{\otimes N}(dx)\\
	&+2\int_0^t \int_{\R^{dN}}\lal \mathbf{b}^N(\nabla \psi^N_s(x))-\mathbf{b}^N(x) , \nabla \psi_t^{N}(x)-x \ral \mu_t^{\otimes N}(dx) \\	
	&+\eta^{-1}\int_0^t \FF^2(\mu_s,b) ds    +\eta\int_0^t W^2_2(\mu_s^{\otimes N},G^N_s) ds.\\
	\end{align*}
	Using \cite[Lemma 3.2]{BGG1} we have for any $x\in \R^{dN}$
	\[
	-2 \lt(\Delta \psi_s^N(x)+\Delta \psi_s^{*,N }(\nabla \psi_s^N(x))-2dN\rt)\leq 0,
	\]
	and therefore
 	\begin{align*}
	W_2^2(\mu_t^{\otimes N},G^N_t)&\leq W_2^2(\mu_0^{\otimes N},G^N_0) +2\int_0^t \int_{\R^{dN}}\lal \mathbf{b}^N(\nabla \psi^N_s(x))-\mathbf{b}^N(x) , \nabla \psi_t^{N}(x)-x \ral \mu_t^{\otimes N}(dx) \\	
	&+\frac{\eta^{-1}}{N^2}\sum_{i\neq j}^N \int_0^t  \int_{\R^{dN}} \lt(b(x_i,x_j)-b*\mu(x_i) \rt)^2\mu_s^{\otimes N}(dx) ds    +\eta\int_0^t W^2_2(\mu_s^{\otimes N},G^N_s) ds.
	\end{align*}
	Taking no advantage of the negative term we have just thrown away, can be seen as analogous to what is done in synchronous coupling. The method is then quite close to the one used in \cite{Gol} in the quantum framework, the main difference being that there it uses the coupling version of the definition of Wasserstein metric, whereas here it uses the optimal transport one. Therefore we can reprove in this way any propagation of chaos result obtained by synchronous coupling.\newline	
	Let us first consider the most classical one, studied in \cite{Szn}, where $b\in Lip(\R^d\times \R^d)\cap L^\infty(\R^d\times\R^d)$ is a bounded Lipschitz field. Hence in this case it is straightforward to obtain
	\begin{align*}
	W_2^2(\mu_t^{\otimes N},G^N_t)&\leq W_2^2(\mu_0^{\otimes N},G^N_0) +(2\|b\|_{Lip}+\eta)\int_0^t W^2_2(\mu_s^{\otimes N},G^N_s) ds+4 \eta^{-1}\|b\|^2_{L^\infty}t, 
	\end{align*}
	and then by Gronwall's inequality
	\begin{align*}
	\frac{W_2^2(\mu_t^{\otimes N},G^N_t)}{N}\leq \lt( \frac{W_2^2(\mu_0^{\otimes N},G^N_0)}{N}+\frac{4 \eta^{-1}\|b\|^2_{L^\infty}}{N}\rt)e^{(2\|b\|_{Lip}+\eta)t}.
	\end{align*}
	We now consider the case studied in \cite{Mal}, of $b(x,y)=-\nabla U(x)-\nabla W(x-y)$ where $W$ is convex, even, with polynomial growth, and $U$ satisfies $\nabla^2 U(x)>\beta I_d$ with $\beta>0$, for any $x\in \R^d$. Then in this case we have by oddness of $\nabla W$
 \begin{align*}
 W_2^2&(\mu_t^{\otimes N},G^N_t)\leq W_2^2(\mu_0^{\otimes N},G^N_0) -2\sum_{i=1}^N\int_0^t \int_{\R^{dN}}\lal \nabla U(\nabla_i \psi^N_s(x))-\nabla U(x_i) , \nabla_i \psi_s^{N}(x)-x_i \ral \mu_t^{\otimes N}(dx) \\
 &-2\sum_{i=1}^N\int_0^t \int_{\R^{dN}}\mej\lal \nabla W(\nabla_i \psi^N_s(x)-\nabla_j\psi^N_s(x))-\nabla W(x_i-x_j) , \nabla_i \psi_s^{N}(x)-\nabla_j\psi_s^N(x)-(x_i-x_j) \ral \mu_t^{\otimes N}(dx) 	\\
 &+\frac{\eta^{-1}}{N^2}\sum_{i\neq j}^N \int_0^t  \int_{\R^{dN}} \lt(\nabla W(x_i-x_j)-\nabla W*\mu(x_i) \rt)^2\mu_s^{\otimes N}(dx) ds    +\eta\int_0^t W^2_2(\mu_s^{\otimes N},G^N_s) ds\\
 &\leq  W_2^2(\mu_0^{\otimes N},G^N_0) -(2\beta-\eta)\int_0^t W^2_2(\mu_s^{\otimes N},G^N_s) ds+\frac{\eta^{-1}}{N^2}\sum_{i\neq j}^N \int_0^t  \int_{\R^{dN}} \lt(\nabla W(x_i-x_j)-\nabla W*\mu(x_i) \rt)^2\mu_s^{\otimes N}(dx) ds. 
 \end{align*}
 Using the polynomial growth assumption on $W$ and assuming uniform in time estimate for moment of some adequate order on $\mu_t$ yields
   \begin{align*}
   W_2^2&(\mu_t^{\otimes N},G^N_t)\leq  W_2^2(\mu_0^{\otimes N},G^N_0) -(2\beta-\eta)\int_0^t W^2_2(\mu_s^{\otimes N},G^N_s) ds+\eta^{-1}Ct
   \end{align*}
    and choosing $\eta<2\beta$ yields the following uniform in time propagation of chaos result
	\begin{align*}
	\frac{W_2^2(\mu_t^{\otimes N},G^N_t)}{N}\leq  \frac{W_2^2(\mu_0^{\otimes N},G^N_0)}{N}e^{-(2\beta-\eta)t}+\frac{1-e^{-(2\eta-\eta)t}}{2\beta-\eta}\frac{\eta^{-1}C}{N}
	\end{align*}
	It could be possible to go on revisiting some synchronous coupling result. But we stop it here and emphasize that this method has no hope to apply in the degenerate diffusion cases, such as hypoelliptic equations. Also it could be interesting to introduce the reverse crossed term in the Proof of Proposition \ref{prop:main} to investigate what could be done with the identity
		\begin{align*}
			W_2^2(G_t^{N},\mu_t^{\otimes N})&=2\int_0^t \int_{\R^{dN}}\lal \mathbf{b}^N(\nabla \psi^N_s(x))-(b*\mu_s)^{\otimes N}(\nabla \psi(x)) , \nabla \psi_t^{N}(x)-x \ral \mu_t^{\otimes N}(dx) \\
			&+2\int_0^t \int_{\R^{dN}}\lal \lt(\nabla \ln G_s^N\rt)(\nabla\psi_t^{N}(x))- \lt(\nabla \ln \mu_s\rt)^{\otimes N}(x) ,  \nabla \psi_t^{N}(x)-x \ral \mu_t^{\otimes N}(dx)\\
			&+2\int_0^t \int_{\R^{dN}}\lal (b*\mu_s)^{\otimes N}(\nabla \psi(x))-(b*\mu_s)^{\otimes N}(x) , \nabla \psi_t^{N}(x)-x \ral \mu_t^{\otimes N}(dx),
		\end{align*}
		which would correspond to what is done, at the particle level, in \cite{Hol} in the case of Holder interaction. Keeping in mind that the dissipation of $W_2^2$ along the heat flow is the relative entropy, it could be interesting to investigate the links with the relative entropy method recently introduced in \cite{JW1,JW2} to see if the method presented here could be adapted to non uniform in time propagation of chaos for the kinetic or singular interaction cases, or the relative entropy method to the uniform in time propagation of chaos for non convex confinement. We delay all these questions to some possible future works.\newline
	
	In all the illustrated examples presented above, the only advantage of the gradient flow approach with respect to the usual synchronous coupling technique, is that it enables more straightforwardly to have the particles system \eqref{eq:PS} starting from initial condition which are else than i.i.d. of law $\mu_0$. \newline
	
	If one wants to obtain some more interesting result, we have to make a better use of the nonnegative term 
	\[
	-2 \lt(\Delta \psi_s^N(x)+\Delta \psi_s^{*,N }(\nabla \psi_s^N(x))-2dN\rt)\leq 0,
	\]
	in the dissipation functional $-\JJ$, that we have thrown away so far. In probabilistic terms, this consideration is the same as the one already discussed in the introduction regarding synchronous and reflecting coupling. Likewise to \cite{Gcou}, it then enables to obtain uniform in time result in the case of convex outside some ball confinement, given in the

	\begin{theorem}
	\label{thm:main}
	Assume that $b(x,y)=-\nabla V(x)-\e\nabla W(x-y)$, with $V(x)=|x|^4-a|x|^2$ and $W(x)=- |x|^2$. There is $a^*>0$ such that if $a\in (0,a^*)$ and $\e\in (0,\e_a)$, $\mu_\infty\in \PP_2(\R^d)$ is the unique stationary solution to \eqref{eq:NLPDE}, and $G_t^N\in \PP^{\textit{sym}}(\R^{dN})$ the law at time $t>0$ of the solution to \eqref{eq:PS} with initial law $G_0^N\in\PP^{\textit{sym}}_6(\R^{dN})\cap L\ln L(\R^{dN})$, then there are  constants $C,\alpha>0$ such that for any $t\geq 0$, $N\geq 2$ it holds	
	\[
	\frac{W_2^2(G_t^N,\mu_\infty^{\otimes N})}{N}\leq \frac{W_2^2(G_0^N,\mu_\infty^{\otimes N})}{N}e^{- \alpha t}+\frac{C}{N}.
	\]
	\end{theorem}

	Note that the maximal depth of the wells $a^*$ comes form the unitary diffusion coefficient in \eqref{eq:NLPDE}. For a potential of given depth of wells, the result could also be obtained if we replace the  diffusion factor $\sqrt{2}$ in \eqref{eq:PS} with $\sqrt{2\sigma}$ for $\sigma$ large enough w.r.t the depth. For $\sigma$ not large enough the limit equation \eqref{eq:NLPDE} is known to admits several stationary solution (see \cite{HT}).\newline
	Also note that Theorem \ref{thm:main}, only provides uniform in time propagation of $\mu_\infty$-chaos with optimal $N^{-1}$ convergence rate. One can obtain  propagation of $\mu_t$-chaos for any $(\mu_t)_{t\geq 0}$ non stationary solution to \eqref{eq:NLPDE}, but with less sharp convergence rate in the 
	\begin{corollary}
		\label{cor:main}
			Let be the assumptions of Theorem \ref{thm:main} be in force. Let be $(\mu_t)_{t\geq 0}$ the solution to \eqref{eq:NLPDE} starting from $\mu_0\in \PP_6(\R^d)\cap L\ln L(\R^d)$, and $G_t^N\in \PP^{\textit{sym}}(\R^{dN})$ be the law at time $t>0$ of the solution to \eqref{eq:PS} with initial law $\mu_0^{\otimes N}$. There are $C>0$ and $\beta\in(0,1)$ such that for any $N\geq 2$ and $t\geq 0$ it holds
			\[
			\frac{W_2^2(\mu_t^{\otimes N},G_t^N)}{N}\leq CN^{-\beta}
			\]

	\end{corollary}


	%
	%
	%
	%
	\section{Optimal transport on $\PP^{\textit{sym}}(\R^{dN})$}
	We begin this section with the elementary 	
	\begin{lemma}
		\label{lem:SADW2}
		Let be $N\geq 2$, $G^N\in \PP_2^{\textit{sym}}(\R^{dN})$ and $F^N\in \PP_2^{\textit{sym}}(\R^d)$. Then for $2 \leq \ell\leq N$ it holds
		\[
		\ell^{-1}W_2^2(G^{N,\ell},F^{N,\ell} )\leq N^{-1} W_2^2(G^{N},F^{N} ),
		\]
		where $G^{N,\ell}$ (resp. $F^{N,\ell}$) is the marginal of $G^N$ (resp. $F^N$) on $\R^{d\ell}$.
	\end{lemma}
	\begin{remark}
		Such a result is likely to be already given somewhere in the literature. Indeed it can be seen roughly speaking as a super additivity property of the \textit{normalized} $W_2^2$ functional on $\PP_2^{\textit{sym}}(\R^{dN})$. But this is well known (see \cite{Car}, \cite{HM,MSN}) for the dissipation of the $W_2^2$ functional along the  heat flow, which is Boltzmann's entropy, as well as for the dissipation of Boltzmann's entropy along the  heat flow which is the Fisher information  (see also \cite{Sal} for the case of the dissipation of the Boltzmann's entropy along the fractional heat flow).\newline
		We will prove this Lemma using probabilist formalism, and it is the only time will do so. Indeed it is more handy when it comes to passing to marginals.		
	\end{remark}
	
	\begin{proof}
		Let be $((X_1,\cdots,X_N),(Y_1,\cdots,Y_N))$ two random vectors on $\R^{dN}$ of respective laws $G^N$ and $F^N$ such that $\LL((X_1,\cdots,X_N),(Y_1,\cdots,Y_N))$ is the optimal coupling between $G^N$ and $F^N$. Then by symmetry 
		\begin{align*}
			W_2^2(G^N,F^N)&=\E\lt[ \sum_{i=1}^N|X_i-Y_i|^2 \rt]=N \E\lt[ |X_1-Y_1|^2 \rt]\\
			&=\frac{N}{\ell} \ell\E\lt[ |X_1-Y_1|^2 \rt]=\frac{N}{\ell}\E\lt[\sum_{i=1}^{\ell} |X_i-Y_i|^2 \rt]\\
			&\geq \frac{N}{\ell} W_2^2(G^{N,\ell},F^{N,\ell}),
		\end{align*}
		since
		$((X_1,\cdots,X_\ell),(Y_1,\cdots,Y_\ell))$ is some coupling between $G^{N,\ell},F^{N,\ell}$.
	\end{proof}
	
	Therefore, if for some sequence $G_N\in \PP_2^{\textit{sym}}(\R^{dN})$ and some probability measure $\mu\in \PP_2(\R^d)$ it holds $\frac{W_2^2(\mu^{\otimes N},G_N)}{N}=o_{N}(1)$, then the sequence $(G_N)$ is $\mu$-chaotic, in the sense of Definition \ref{def:chaos}.\newline	

	We provide the result which enables to adapt the techniques of \cite{BGG1,BGG2} at the particle level in the
	\begin{proposition}
		\label{prop:use}
		Let be $F^N,G^N\in \PP^{\textit{sym}}_2(\R^{dN})$ two symmetric probabilities and let be $\psi^N\in\CC^2(\R^{dN})$ be the m.K.p. associated to the couple $(F^N,G^N)$ and $\psi^{N*}$ its convex conjugate. Then for any $x\in \R^{dN}$ it holds
		\[
		\Delta	 \psi^{N*}(\nabla\psi^N(x))\geq \sum_{i=1}^N Tr\lt[ (\nabla^2_{i,i}\psi^N(x))^{-1}\rt]
		\]
	\end{proposition}

	\begin{proof}
		First observe that by definition of $\psi^N$ it holds
		\[
		\nabla \psi^{N*}\lt( \nabla \psi^N(x) \rt)=x.
		\]	
		So that
		\[
		\nabla^2 \psi^{N*}\lt( \nabla \psi^N(x) \rt)\nabla\psi^N(x)=I_{dN}.
		\]
		Since $\nabla^2 \psi^N(x)$ is on the one hand symmetric by Schwarz's Theorem since $\psi^N\in \CC^2(\R^{dN})$, and invertible on the other hand, $\nabla^2\psi^N(x)$ is symmetric positive definite for any $x\in \R^{dN}$. Moreover 
		\[
		\Delta	 \psi^{N*}(\nabla\psi^N(x))=Tr\lt[\nabla^2 \psi^{N*}\lt( \nabla \psi^N(x) \rt) \rt]=Tr\lt[( \nabla^2\psi^N(x))^{-1}  \rt],
		\]
		and we conclude by using Proposition \ref{prop:sym}.
	\end{proof}
		
%
%
%
%

\section{Proof of Theorem \ref{thm:main} and Corollary \ref{thm:main}}

In this section, we set $b(x,y)=-\nabla V(x)-\e\nabla W(x-y)$ with
\bq
\label{eq:doublewell}
V(x)=|x|^4-a|x|^2, \ W(x)=-|x|^2
\eq
In order to prove the claimed uniform in time propagation of chaos result, we first need to extend the techniques of \cite{BGG1} used in the case of degenerately convex confinements to the case of non convex ones. We begin by the

\begin{lemma}
	\label{lem:use}
	The potentials $V,W$ satisfy
	\begin{itemize}
	\item[$(i)$]
	For any $(x,y)\in \R^{2d}$ 
	\[
	-\lal\nabla V(x)-\nabla V(y),x-y\ral \leq 2a|x-y|^2
	\]
	\item[$(ii)$]
	For any $R>\lt(\frac{a}{6}\rt)^{1/2}$, and $(x,y)\notin \B_{2R}\times \B_{2R}$
	\[
	\lal\nabla V(x)-\nabla V(y),x-y\ral \geq  4(R^2-\frac{a}{6})|x-y|^2,
	\]
	\item[$(iii)$] for any $\mu\in \PP_2(\R^d)$ and $R>0$
		\[
		\| (V+\e W*\mu)_{|\B_{3R} } \|_{L^\infty}\leq (3R)^4+(a+2\e)(3R)^2+ 2\e\int_{\R^d}|z|^2\mu(dz).
		\]
\end{itemize}
\end{lemma}

\begin{proof}
	
	Observe that
	\[
	\nabla V(x)=(4|x|^2-2a)x
	\]

	\begin{align*}
	-\lal\nabla V(x)-\nabla V(y),x-y\ral&=-\lal (4|x|^2-2a)x-(4|y|^2-2a)y,x-y\ral\\
	&=-4\lal |x|^2x-|y|^2y,x-y\ral+2a|x-y|^2,
	\end{align*}
	using that
	\begin{align*}
	\lal |x|^2x-|y|^2y,x-y\ral= & |x|^4-\lal x, y\ral|x|^2-\lal x, y\ral|y|^2+|y|^4	\\
	&\geq (|x|^2+|x| |y| +|y|^2 )(|x|-|y|)^2\geq 0,\\
	\end{align*}
	concludes the proof of the first point. Then
	\begin{align*}
	\nabla^2 V(x)=&8x\otimes x + (4|x|^2-2a)I_d\\
	&=8|x|^2\lt( \frac{x\otimes x}{|x|^2}-I_d  \rt)+12(|x|^2-\frac{a}{6})I_d,
	\end{align*}
	so that for any $|x|>R$ it holds
	\begin{align*}
	\nabla^2 V(x) > 12(R^2-\frac{a}{6})I_d.
	\end{align*}
	We use \cite[Lemma 3.12]{BGG1} to conclude the proof of the second point. Then observe that
	\begin{align*}
	-V(x)-\e W*\mu(x)=&-|x|^4+a|x|^2+ \e \int_{\R^d} |x-z|^2\mu(dz)\\
	&=-|x|^4+(a+ \e)|x|^2+ 2\e x \int_{\R^d}z\mu(dz)+\e\int_{\R^d}|z|^2\mu(dz),
	\end{align*}
	and then for any $|x|\leq 3R$
	\begin{align*}
	\lt|V(x)+\e W*\mu(x)\rt|&\leq 	|x|^4+(a+ 2\e)|x|^2+2\e\int_{\R^d}|z|^2\mu(dz)\\
	&\leq (3R)^4+(a+ 2\e)(3R)^2+2\e\int_{\R^d}|z|^2\mu(dz),
	\end{align*}	
	so that the result is proved.

\end{proof}	
Next we need some moments estimates given in the 
\begin{lemma}
	\label{lem:mom}
	Let be $(\mu_t)_{t\geq 0}$ a solution to \eqref{eq:NLPDE} strating from $\mu_0\in \PP_2(\R^d)$. Then for any $t,\delta>0$ it holds
	\begin{align*}
		\int_{\R^d} |x|^2\mu_t(dx)\leq e^{-4\delta t}\int_{\R^d} |x|^2\mu_0(dx)+\frac{(a+\e+\delta)^2+d}{4\delta}
	\end{align*}
	In particular if $\mu_\infty$ is a stationary solution to \eqref{eq:NLPDE}, then it holds
	\[
	\int_{\R^d} |x|^2\mu_\infty(dx)\leq \sqrt{(a+\e)^2+d}.
	\]
	Moreover if $\mu_0\in \PP_k(\R^d)$ for some $k \geq 2$ then
	\[
	\sup_{t>0}\int_{\R^{d}}|x|^k\mu_t(dx)<\infty.
	\]
	Similarly if $(G_t^N)_{t\geq 0}$ is the solution to \eqref{eq:Loui} starting from $G_0^N\in \PP^{\text{sym}}_k(\R^{dN})$ then
	\[
	\sup_{t>0}\int_{\R^{dN}}|x_1|^kG_t^N(dx)<\infty.
	\]
\end{lemma}
\begin{proof}
	Let be $(\mu_t)_{t\geq 0}$ be a solution to \eqref{eq:NLPDE} then
	\begin{align*}
		\frac{d}{dt}\frac{1}{2}\int_{\R^d} |x|^2\mu_t(dx)=&-\int_{\R^d} x\cdot\lt( \nabla V(x)+\e\nabla W*\mu_t(x)\rt) \mu_t(dx) 	+d\int_{\R^d} \mu_t(dx)\\
		&=-\int_{\R^d} (4|x|^2-2a)|x|^2\mu_t(dx)+\e\int_{\R^{2d}}x\cdot(x-z)\mu_t(dz)\mu_t(dx)+d\\
		&=-\int_{\R^d} (4|x|^2-2a- \e)|x|^2\mu_t(dx)+\e\lt(\int_{\R^d} x \mu(dx)\rt)^2+d\\
		&\leq -\int_{\R^d} (4|x|^2-2a- 2\e)|x|^2\mu_t(dx)+d\\
	\end{align*}
	So that for any $\delta>0$ it holds
	\begin{align*}
		\frac{d}{dt}\frac{1}{2}\int_{\R^d} |x|^2\mu_t(dx)&\leq -\int_{\R^d} (4|x|^2-2a- 2\e)|x|^2\mu_t(dx)+d\\
		&\leq - 2\delta \int_{\R^d} \mb_{|x|> \sqrt{\frac{a+\e+\delta}{2}}}|x|^2\mu_t(dx) +2(a+\e) \int_{\R^d} \mb_{|x|\leq \sqrt{\frac{a+\e+\delta}{2}}}|x|^2\mu_t(dx)+d \\
		&=-2\delta \int_{\R^d} |x|^2\mu_t(dx)+2\lt(\delta+a+\e\rt)\int_{\R^d} \mb_{|x|\leq \sqrt{\frac{a+\e+\delta}{2}}}|x|^2\mu_t(dx)+d\\
		&\leq -2\delta\int_{\R^d} |x|^2\mu_t(dx)+(a+\e+\delta)^2+d
	\end{align*}
	
	Which yields
	\begin{align*}
		\int_{\R^d} |x|^2\mu_t(dx)\leq e^{-4\delta t}\int_{\R^d} |x|^2\mu_0(dx)+\frac{(a+\e+\delta)^2+d}{4\delta}
	\end{align*}
	In particular for any $\delta>0$
	\begin{align*}
		\int_{\R^d} |x|^2\mu_\infty(dx)\leq\frac{(a+\e+\delta)^2+d}{4\delta}.
	\end{align*}
	We minimize the r.h.s. by choosing $\delta=\sqrt{(a+\e)^2+d}$ which yields the desired result. \newline
	Then for $k\geq 2$
	 \begin{align*}
	 	\frac{d}{dt}\frac{1}{k}\int_{\R^d} |x|^k\mu_t(dx)=&-\int_{\R^d} |x|^{k-2} x\cdot\lt( \nabla V(x)+\e\nabla W*\mu_t(x)\rt) \mu_t(dx) 	+(d+k-2)\int_{\R^d} |x|^{k-2}\mu_t(dx)\\
	 	&=-\int_{\R^d} (4|x|^2-2a)|x|^k\mu_t(dx)+\e\int_{\R^{2d}}|x|^{k-2}x\cdot(x-z)\mu_t(dz)\mu_t(dx)\\
	 	&+(d+k-2)\lt(\int_{\R^d} |x|^{k}\mu_t(dx)\rt)^{\frac{k-2}{k}}\\
	 	&=-\int_{\R^d} (4|x|^2-2a- \e)|x|^k\mu_t(dx)+\e\lt(\int_{\R^d} |x|^{k-2}x \mu(dx)\rt)\lt(\int_{\R^d} z \mu(dx)\rt)\\
	 	&+(d+k-2)\lt(\int_{\R^d} |x|^{k}\mu_t(dx)\rt)^{\frac{k-2}{k}}\\
	 	&\leq -\int_{\R^d} (4|x|^2-2a- 2\e)|x|^k\mu_t(dx)+(d+k-2)\lt(\int_{\R^d} |x|^{k}\mu_t(dx)\rt)^{\frac{k-2}{k}}\\
	 \end{align*}
	 Using a similar argument as above and Young's inequality again yields for some constants $c,C>0$
	 \begin{align*}
	 	\frac{d}{dt}\int_{\R^d} |x|^k\mu_t(dx)&\leq -c\int_{\R^d} |x|^k\mu_t(dx)+C,
	 \end{align*}
	 
	 Which  using Gronwall's inequality
	 \begin{align*}
	 	\sup_{t>0}\int_{\R^d} |x|^k\mu_t(dx)<\infty.
	 \end{align*}
	  Then by Ito's rule, it holds
	 \begin{align*}
	 	|X_t^i|^k=&|X_0^i|^k-k\int_0^t |X_s^i|^{k-2}X_s^i\cdot \lt( \nabla V(X_s^i) +\frac{\e}{N}\sum_{j=1}^N\nabla W(X_s^i-X_s^j) \rt)ds+k\sqrt{2}\int_0^t |X_s^i|^{k-2}X_s^i\cdot dB_s^i\\
	 	&+(d+k-2)\int_0^t |X_s^i|^{k-2}ds,
	 \end{align*}	
	 so that taking the expectation and averaging over $i=1,\cdots,N$, we obtain by symmetry 
	  \begin{align*}
	  	\int_{\R^{dN}}|x_1|^kG_t^N(dx)&=\mei \E\lt[|X_t^i|^k\rt]\\
	  	&=\mei \E\lt[|X_0^i|^k\rt]-k\int_0^t\mei\E\lt[ |X_s^i|^{k-2}X_s^i\cdot \lt( \nabla V(X_s^i) +\frac{\e}{N}\sum_{j=1}^N\nabla W(X_s^i-X_s^j) \rt)\rt]ds\\
	  	&+(d+k-2)\int_0^t \E\lt[|X_s^i|^{k-2}\rt]ds.
	  \end{align*}	
	  and using symmetry and similar arguments as the one used above yields
	  \[
	  \sup_{t>0}\int_{\R^{dN}}|x_1|^kG_t^N(dx)<\infty.
	  \]
\end{proof}

Before giving the main Proposition of this Section, we need to look at the existence of stationary solution to \eqref{eq:NLPDE}. Consider the functional $F$ defined on $\PP_2(\R^d)$ as
\[
F(\mu)=\int_{\R^d} \mu\ln \mu + \int_{\R^d} V(x)\mu (dx)+\frac{\e}{2}\int_{\R^{2d}}W(x-y)\mu(dx)\mu(dy).
\]
We use \cite[Proposition 4.4, point (iii)]{BGG2} to deduce that there is $\mu_\infty\in \PP_2(\R^d)$ which minimizes $F$, and that such a measure is a stationary solution to \eqref{eq:NLPDE} and satisfies
\[
\mu_\infty(x)=\lt(\int_{\R^d}e^{-V(z)-\e W*\mu_\infty(z)}dz\rt)^{-1}e^{-V(x)-\e W*\mu_\infty(x)}=:Z^{-1}e^{-V_\e(x)}.
\]
We finally need to state the

\begin{lemma}
	\label{lem:cst}
	Let be $a,\e>0$ and for any $R>0$ define
		\begin{align*}
		&\CC_1(R,a,\e):=\lt(36R^2e^{2\sup_{|z|\leq 3R}V_\e(z)}\rt)^{-1}-2a,\\
		&\CC_2(R,a,\e):=4( R^2-\frac{a}{6} )-4\lt(36R^2\rt)^{-1}.
		\end{align*}
		There is $a^*>0$ such that if $a\in(0,a^*)$ and $\e\in(0,a/2)$, there are $\kappa_a>0$ and $R_a>\sqrt{\frac{a}{6}}$ such that	
		\begin{align*}
		\CC_1(R_a,a,\e)\wedge \CC_2(R_a,a,\e)>\kappa_a,
		\end{align*}
	\end{lemma}

\begin{proof}
	We begin by observing that by definition of $\mu_\infty$, point $(iii)$ of Lemma \ref{lem:use} and Lemma \ref{lem:mom} it holds
	\begin{align*}
	\sup_{|z|\leq 3R} |V_\e(z)|&\leq (3R)^4+(a+2\e)(3R)^2+2\e\int |z|^2\mu_\infty(dz)\\
	&\leq (3R)^4+(a+2\e)(3R)^2+2\e((a+\e)+\sqrt{d})\\
	&\leq (3R)^4+6(a/6+2\e/6)(3R)^2+72\e/6(a/6+\e/6)+2\sqrt{d}\e.
	\end{align*}
	Since $\e<a/2$, for any $R>\sqrt{\frac{a}{6}}$ it holds for some $C>0$
	\begin{align*}
	\sup_{|z|\leq 3R} |V_\e(z)|&\leq  C(R^4+R^2).
	\end{align*}
	For $a>0$, define
	\[
	R_a=\inf\{ R>0, \  ( R^2-\frac{a}{6} )>\lt( 36 R^2 \rt)^{-1} +a \}
	\]	
	So that it holds $\CC_2(R_a,a,\e)>4a$. Then by the observation of the beginning of the proof it holds
	\begin{align*}
		\CC_1(R,a,\e)&\geq \lt(36R^2\rt)^{-1}e^{-2C(R^4+R^2)}-2a
	\end{align*}	 
	we define the function $g(a)=\lt(36R_a^2\rt)^{-1}e^{-2 C(R_a^4+R_a^2) }-2a$. Since $g$ is continuous and $g(0)>0$, there is some $a^*>0$ such that for any $a\in (0,a^*)$,  $\CC_1(R_a,a,\e)\geq g(a)>0$, and the result is proved with $\kappa_a=(4a)\wedge g(a)$.

\end{proof}	

We can now give the main result of this Section, which from which will follow the main Theorems of this paper, in the

\begin{proposition}
	\label{prop:WJ}
	There is $a^*>0$, and for any $a\in(0,a^*)$, $\e_a>0$, such that if $\e\in(0,\e_a)$ and $\mu_\infty$ is a stationary solution to \eqref{eq:NLPDE} , then there is a constant $\kappa>0$ such that for any $N \geq 2$, $G^N\in \PP_2^{\text{sym}}(\R^{dN})$ it holds 
	\[
	\kappa W_2^2(G^N,\mu_\infty^{\otimes N})\leq \JJ(G^N|\mathbf{b}^N,\mu_\infty^{\otimes N}),
	\]
	i.e. $\mu_\infty$ satisfies a symmetric $WJ(\kappa)$ inequality.	
\end{proposition}
Most of the material of the proof of this Proposition is taken from \cite[Proposition 3.4]{BGG1}. Nevertheless we will write it completely. Because on the one hand we will need to make the constants in the estimates more precise since we are in a  non convex coefficients framework. And on the other hand, because we want to emphasize what differs from the classical and the $\PP^{\textit{sym}}(\R^{dN})$ contexts.

\begin{proof}
First recall that	
	\begin{align*}
		\JJ(G^N|\mathbf{b}^N,\mu^{\otimes N}_\infty) &= \int_{\R^{dN}}\lt(\Delta \psi_s^{N}(x)+\Delta \psi_s^{*N}(\nabla \psi_s^N(x))-2dN\rt) \mu^{\otimes N}_\infty(dx) ds\\
		&+\sum_{i=1}^N\int_{\R^{dN}} \lal \nabla V(\nabla_i \psi^N(x))-\nabla V(x_i)   , \nabla_i \psi^N(x)-x_i  
		\ral\mu^{\otimes N}_\infty(dx)\\
		&+\e\sum_{i=1}^N\int_{\R^{dN}} \lal \mej \nabla W(\nabla_i \psi^N(x)-\nabla_j\psi^{N}(x))-\nabla W(x_i-x_j)   , \nabla \psi_i^N(x)-x_i  \ral\mu^{\otimes N}_\infty(dx)\\
		&:=\mt^1+\mt^2+\mt^3
	\end{align*}	
	Recall that
	\[
	W_2^2(G^N,\mu^{\otimes N}_\infty)=\int_{\R^{dN}}\sum_{i=1}^N\lt|\nabla_i \psi^N(x)-x_i\rt|^2\mu^{\otimes N}_\infty
	\]
	
	\textit{Step one} : Estimate of $\mt^1$ \newline

	First we easily obtain
	\begin{align*}
	\mt^3&\geq -\frac{\e }{N}\sum_{i,j=1}^N\int_{\R^{dN}}\lt|\nabla_i \psi^N(x)-\nabla_j\psi^{N}(x))-(x_i-x_j)   \rt|\lt| \nabla \psi_i^N(x)-x_i \rt|\mu^{\otimes N}_\infty(dx)\\
	&\geq -4\e \sum_{i=1}^N\int_{\R^{dN}}\lt| \nabla \psi_i^N(x)-x_i \rt|^2\mu^{\otimes N}_\infty(dx)=-4\e W_2^2(G^N,\mu^{\otimes N}_\infty)
	\end{align*}
	
	\textit{Step two} : Estimate of $\mt^2$ \newline
	
	Let be $R>\lt(\frac{a}{6}\rt)^{1/2}$ to be fixed later.
	We denote for $i=1,\cdots,N$
	\[
	\XX_i^N=\{ x\in \R^{dN}, |\nabla \psi_i^N(x)|\leq 2R, \quad |x_i|\leq 2R   \}
	\]
	Then by points $(i),(ii)$ of Lemma \ref{lem:use} we have
	\begin{align*}
		\int_{\R^{dN}}\lal \nabla V(\nabla_i \psi^N(x))-\nabla V(x_i)   , \nabla_i \psi^N(x)-x_i \ral\mu^{\otimes N}_\infty&=\int_{\R^{dN}\setminus \XX_i^N }\lal \nabla V(\nabla_i \psi^N(x))-\nabla V(x_i)   , \nabla_i \psi^N(x)-x_i \ral\mu^{\otimes N}_\infty\\
		&+\int_{\XX_i^N }\lal \nabla V(\nabla_i \psi^N(x))-\nabla V(x_i)   , \nabla_i \psi^N(x)-x_i \ral\mu^{\otimes N}_\infty\\
		&\geq 4( R^3-\frac{a}{6} )\int_{\R^{dN}\setminus \XX_i^N }\lt|\nabla_i \psi^N(x)-x_i \rt|^2\mu^{\otimes N}_\infty\\
		&-2a\int_{\XX_i^N }\lt| \nabla_i \psi^N(x)-x_i \rt|^2\mu^{\otimes N}_\infty
	\end{align*}
	
	\textit{Step three} :  Estimate of $\mt^1$ \newline
	For the rest of the proof for $x=(x_1,\cdots,x_N)\in\R^{dN}$, $z\in \R^d$ we will denote
	\begin{align*}
	&x'_i=(x_1,\cdots,x_{i-1},x_{i+1},\cdots,x_N)\in \R^{d(N-1)}\\
	&\tilde{x}^i_z=(x_1,\cdots,x_{i-1},z,x_{i+1},\cdots,x_N)\in \R^{dN}
	\end{align*}

	 We begin by rewriting
	\begin{align*}
	\int_{\XX_i^N}&\lt|\nabla_i\psi^N(x)-x_i\rt|^2\mu_\infty^{\otimes N}(dx)=\alpha_dZ^{-1}\int_{\R^{d(N-1)}\times\mathbb{S}^{d-1}\times(0,\infty)} \mb_{\tilde{x}^i_{r\theta}\in \XX_i^N}|\nabla_i\psi^N(\tilde{x}^i_{r\theta})-r\theta|^2 e^{-V_\e(r\theta)} r^{d-1}drd\theta \mu_\infty^{\otimes (N-1)}(dx'_i)  \\
	\end{align*}
	We will write the rest of the proof as if $\alpha_d=Z$ to make the notations lighter.	Consider now $i=1,\cdots,N$, $(x_1,\cdots,x_{i-1},x_{i+1},\cdots,x_N)\in\R^{d(N-1)}$ and $\theta \in \mathbb{S}^{d-1}$ fixed and define
	\[
	R^i_\theta:=\sup\{ r\geq 0, \tilde{x}^i_{r\theta}\in \XX_i^N \}
	\]
	and $r^i_\theta>0$ as 
	\[
	|\nabla_i\psi^N(\tilde{x}^i_{r^i_\theta\theta})-r^i_\theta\theta|=\inf_{r\in [R^i_\theta,3R]}\{ \lt|\nabla_i\psi^N(\tilde{x}^i_{r\theta})-r\theta\rt|  \},
	\]
	 Then for any $r\in[0,R^i_\theta]$ it holds
	\begin{align*}
		\nabla_i\psi^N(\tilde{x}^i_{r\theta})-r\theta=&\nabla_i\psi^N(\tilde{x}^i_{r^i_\theta\theta})-r_\theta^i\theta+(r-r^i_\theta)\int_0^1  \lt(\nabla_{i,i}^2\psi^N(x^i_s)-I_{d}\rt) \cdot\theta ds\\
	\end{align*}
	with
	\begin{align*}
	x^i_s=&\tilde{x}^i_{r\theta}+s(\tilde{x}^i_{r^i_\theta\theta}-\tilde{x}^i_{r\theta})\\
	&=(x_1,\cdots,x_{i-1},(r+s(r-r^i_\theta))\theta,x_{i+1},\cdots,x_N)
	\end{align*}
	Hence
	\begin{align*}
		\int_{\XX_i^N}\lt|\nabla_i\psi^N(x)-x_i\rt|^2&\mu_\infty^{\otimes N}(dx) \\
		&\leq 2\int_{\R^{d(N-1)}}\int_{\mathbb{S}^{d-1}} \int_0^{\infty}\mb_{\tilde{x}^i_{r\theta}\in \XX_i^N} \lt|\nabla_i\psi^N(\tilde{x}^i_{r^i_\theta\theta})-r_\theta^i\theta\rt|^2 e^{-V_\e(r\theta)} r^{d-1}drd\theta \mu_\infty^{\otimes(N-1)}(dx'_i)\\
		&+2\int_{\R^{d(N-1)}}\int_{\mathbb{S}^{d-1}} \int_0^{\infty} \mb_{\tilde{x}^i_{r\theta}\in \XX_i^N}\lt|(r-r^i_\theta)\int_0^1  \lt(\nabla_{i,i}^2\psi^N(x^i_s)-I_{d}\rt) \cdot\theta ds\rt|^2 e^{-V_\e(r\theta)} r^{d-1}drd\theta \mu_\infty^{\otimes(N-1)}(dx'_i)\\
		&=2(\RR_1^i+\RR_2^i)
	\end{align*}
	$\bullet$ Estimate of $\RR_2^i$\newline
	
	First we use Holder's inequality to obtain
	\begin{align*}
	|(r-r^i_\theta)\int_0^1  \lt(\nabla_{i,i}^2\psi^N(x^i_s)-I_{d}\rt) \cdot\theta ds|^2& =	\lt|(r-r^i_\theta)\int_0^1  \lt((\nabla_{i,i}^2\psi^N(x^i_s))^{1/2}-(\nabla_{i,i}^2\psi^N(x^i_s))^{-1/2}\rt)(\nabla_{i,i}^2\psi^N(x^i_s))^{1/2} \cdot\theta ds\rt|^2\\
	&\leq(r-r^i_\theta)  \int_0^1 Tr\lt[ \lt((\nabla_{i,i}^2\psi^N(x^i_s))^{1/2}-(\nabla_{i,i}^2\psi^N(x^i_s))^{-1/2}\rt)^2  \rt]ds \\
	&\times(r-r^i_\theta) \int_0^1  \lt|(\nabla_{i,i}^2\psi^N(x^i_s))^{1/2} \cdot \theta\rt|^2 ds	
	\end{align*}
	but since
	\begin{align*}
		(r-r^i_\theta)\lt|(\nabla_{i,i}^2\psi^N(x^i_s))^{1/2} \cdot\theta\rt|^2=\lt((r-r^i_\theta)\nabla_{i,i}^2\psi^N(x^i_s)\cdot \theta\rt)\cdot\theta
	\end{align*}	
	we have
		\begin{align*}
			(r-r^i_\theta)\int_0^1  \lt|(\nabla_{i,i}^2\psi^N(x^i_s))^{1/2} \cdot \theta\rt|^2 ds&=\lt(\int_0^1\nabla_{i,i}^2\psi^N(x^i_s)\cdot (r-r^i_\theta) \theta ds\rt)\cdot \theta \\
			&=\lt( \nabla_i \psi^{N}(\tilde{x}^i_{r\theta})-\nabla_i\psi^N(\tilde{x}^i_{r^i_\theta\theta}) \rt)\cdot \theta,
		\end{align*}	
	and for $r>0$ such that $\tilde{x}^i_{r\theta}\in \XX_i^N$ we have by definition $\lt|\nabla_i\psi^N(\tilde{x}^i_{r\theta})\rt|\leq 2R$. Moreover, by definition of $\tilde{x}^i_{r^i_\theta\theta}$ and $R^i_\theta$ it holds
	\begin{align*}
		\lt|\nabla_i\psi^N(\tilde{x}^i_{r^i_\theta\theta}) \rt|&\leq\lt|\nabla_i\psi^N(\tilde{x}^i_{r^i_\theta\theta}) -r_\theta^i \theta \rt|+r_\theta^i\\
		&\leq \lt|\nabla_i\psi^N(\tilde{x}^i_{R^i_\theta\theta}) - R^i_\theta \theta \rt|+3R\\
		&\leq \lt|\nabla_i\psi^N(\tilde{x}^i_{R^i_\theta\theta}) \rt|+R^i_\theta +3R\leq 7R.
	\end{align*}
	Therefore
	\begin{align*}
	\RR^2_i&\leq 9R\int_{\R^{d(N-1)}}\int_{\mathbb{S}^{d-1}} \int_0^{\infty} \mb_{\tilde{x}^i_{r\theta}}(r-r^i_\theta) \int_0^1 D_i(x_s^i)ds e^{-V_\e(r\theta)} r^{d-1}drd\theta \mu_\infty^{\otimes(N-1)}(dx_i')\\
	\end{align*}
	with
	\begin{align*}
	D_i(x)&=Tr\lt[  \lt((\nabla_{i,i}^2\psi^N(x))^{1/2}-(\nabla_{i,i}^2\psi^N(x))^{-1/2}\rt)^2\rt]\\
	&=Tr\lt[  \nabla_{i,i}^2\psi^N(x)+(\nabla_{i,i}^2\psi^N(x))^{-1}-2I_d\rt]
	\end{align*}
	Then by definition of $R^i_\theta$ 
	\begin{align*}
	\int_{\mathbb{S}^{d-1}}&\int_0^{\infty}\mb_{x\in \XX_i^N}(r-r^i_\theta) \int_0^1 
	D_i(x_s^i)ds e^{-V_\e(r\theta)}r^{d-1}drd\theta\\
	&=\int_{\mathbb{S}^{d-1}}\int_0^{R^i_\theta} (r-r^i_\theta) \int_0^1 D_i(x_s^i)e^{-V_\e ((r+s(r_\theta^i-r))\theta)}e^{V_\e ((r+s(r_\theta^i-r))\theta)}ds r^{d-1} e^{-V_\e(r\theta)}drd\theta\\	
	&= e^{2\sup_{|z|\leq 3R}|V_\e(z)|} \int_0^{R^i_\theta}  \int_{\mathbb{S}^{d-1}}\int_r^{r_\theta^i}D_i(\tilde{x}^i_{u\theta})u^{d-1} e^{-V_\e(u\theta)} du d\theta dr\\
	&\leq 2Re^{\sup_{|z|\leq 3R}|V_\e(z)|}\int_{\R^d}\mb_{\B_{3R}}(z)D_i(\tilde{x}^i_z)\mu_\infty(dz)
	\end{align*}
so that
	\begin{align*}
	&\RR^2_i\leq 18R^2e^{2\sup_{|z|\leq 3R}|V_\e(z)|} \int_{\R^{dN}}Tr\lt[  \nabla_{i,i}^2\psi^N(\tilde{x}^i_z)+(\nabla_{i,i}^2\psi^N(\tilde{x}^i_z))^{-1}-2I_d\rt]e^{-V_\e(z)}dz\mu_\infty^{\otimes (N-1)}(dx'_i) \\
	&\leq  18R^2e^{2\sup_{|z|\leq 3R}|V_\e(z)|}  \int_{\R^{dN}}Tr\lt[  \nabla_{i,i}^2\psi^N(x)+(\nabla_{i,i}^2\psi^N(x))^{-1}-2I_d\rt]\mu_\infty^{\otimes N}(dx)
	\end{align*}	
	Finally using Proposition \ref{prop:use}
	\begin{align*}
		\sum_{i=1}^N\RR_i^2&\leq 18R^2e^{2\sup_{|z|\leq 3R}|V_\e(z)|}  \int_{\R^{dN}} \lt(\Delta \psi^N(x)+\Delta\psi^{*N}(\nabla\psi^N(x))-2dN\rt) \mu_\infty^{\otimes N}(dx)
	\end{align*}	
	
	$\bullet$ Estimate of $\RR^1_i$\newline
	By the respective definitions of $R_\theta^i$ and $r_\theta^i$ it holds
	\begin{align*}
		\int_{\mathbb{S}^{d-1}} \int_0^{\infty}\mb_{\tilde{x}^i_{r\theta}\in \XX_i^N}& |\nabla_i\psi^N(\tilde{x}^i_{r_\theta^i\theta})-r_\theta^i\theta|^2 e^{-V_\e(r\theta)} r^{d-1}drd\theta  \leq \int_{\mathbb{S}^{d-1}} \int_0^{2R} |\nabla_i\psi^N(\tilde{x}^i_{r_\theta^i\theta})-r_\theta^i\theta|^2 e^{-V_\e(r\theta)} r^{d-1}drd\theta \\
		&\leq \int_{\mathbb{S}^{d-1}} \int_0^{2R} R^{-1}\int_{2R}^{3R} |\nabla_i\psi^N(\tilde{x}^i_{u\theta})-u\theta|^2 du e^{-V_\e(r\theta)} r^{d-1}dr d\theta\\
		&\leq \int_0^{2R} \ R^{-1}\int_{\mathbb{S}^{d-1}}\int_{2R}^{3R} |\nabla_i\psi^N(\tilde{x}^i_{u\theta})-u\theta|^2  e^{V_\e(u\theta)}e^{-V_\e(u\theta)} u^{d-1} du d\theta e^{-V_\e(r\theta)} dr \\
		&\leq 2e^{2\sup_{z\in \B_{3R}} |V_\e(z)| }  \int_{\R^d} \mb_{\B_{3R}\setminus\B_{2R}}(z)|\nabla_i\psi^N(\tilde{x}_z^i)-z|^2  \mu_\infty(dz)
	\end{align*}	
	which yields by definition of $\XX_i^N$
	\begin{align*}
		\sum_{i=1}^N\RR^1_i&\leq2e^{2\sup_{z\in \B_{3R}} |V_\e(z)| } \sum_{i=1}^N \int_{\R^{dN}} \mb_{\B_{3R}\setminus\B_{2R}}(x_i)|\nabla_i\psi^N(x)-x_i|^2  \mu^{\otimes N}_\infty(dx)\\
		&\leq 2e^{2\sup_{z\in \B_{3R}} |V_\e(z)| } \sum_{i=1}^N \int_{\R^{dN}\setminus \XX_i^N}|\nabla_i\psi^N(x)-x_i|^2  \mu^{\otimes N}_\infty(dx)
	\end{align*}	
	
	Gathering all these estimates yields
		\begin{align*}
		\mt^1\geq \lt(36R^2e^{2\sup_{|z|\leq 3R}|V_\e(z)|}\rt)^{-1} 	\sum_{i=1}^N\int_{\XX_i^N}\lt|\nabla_i \psi^N(x)-x_i\rt|^2\mu_\infty^{\otimes N}(dx)-4\lt(36R^2\rt)^{-1}  \sum_{i=1}^N\int_{\R^{dN}\setminus \XX_i^N}  \lt|\nabla_i \psi^N(x)-x_i\rt|^2 \mu_\infty^{\otimes N}(dx)
		\end{align*}
	
	\textit{Step four} : Conclusion. \newline

	Gathering the results obtained in the above steps yields
	
	\begin{align*}
	\JJ(G^N|V,W, \mu_\infty^{\otimes N} )&\geq -4\e W_2^2(G^N,\mu^{\otimes N}_\infty)+ \int_{\R^{dN}} \lt(\Delta \psi^N(x)+\Delta\psi^{*N}(\nabla\psi^N(x))-2dN\rt) \mu_\infty^{\otimes N}(dx) \\
	&+4( R^2-\frac{a}{6} ) \sum_{i=1}^N\int_{\R^{dN}\setminus \XX_i^N }\lt|\nabla_i \psi^N(x)-x_i \rt|^2\mu^{\otimes N}_\infty-2a\sum_{i=1}^N\int_{\XX_i^N }\lt| \nabla_i \psi^N(x)-x_i \rt|^2\mu^{\otimes N}_\infty\\
	&\geq \CC_1(R,a,\e)\sum_{i=1}^N\int_{\XX_i^N}\lt|\nabla_i \psi^N(x)-x_i\rt|^2\mu_\infty^{\otimes N}(dx)+\CC_2(R,a,\e) \sum_{i=1}^N\int_{\R^{dN}\setminus \XX_i^N }\lt|\nabla_i \psi^N(x)-x_i \rt|^2\mu^{\otimes N}_\infty\\
	&-4\e W_2^2(G^N,\mu^{\otimes N}_\infty)\\
	&\geq \lt(\lt( \CC_1(R,a,\e) \wedge \CC_2(R,a,\e) \rt)-4\e \rt)W_2^2(G^N,\mu^{\otimes N}_\infty)
	\end{align*}
	 We finally choose $R=R_a$ as given by Lemma \ref{lem:cst} so that there is $\kappa_{a}>0$ such that for any $\e \in (0,a/2)$ it holds
	 \[
	 \lt( \CC_1(R_a,a,\e) \wedge \CC_2(R_a,a,\e) \rt)>\kappa_{a}.
	 \] 
	We then define $\e_a= (\kappa_{a}/4)\wedge (a/2)$ and $\kappa=\kappa_a-4\e$ and the result is proved.
\end{proof}

\textit{Proof of Theorem \ref{thm:main} and Corollary \ref{cor:main}}

We begin by checking that the technical assumptions are fulfilled. Let be $(G_t^N)_{t\geq 0}$ be the solution to \eqref{eq:Loui} with $b(x,y)=-\nabla V(x)-\e \nabla W(x-y)$, and $\mu_\infty$ be a stationary solution to \eqref{eq:NLPDE}. First the vector fields $\nabla V+\e\nabla W*\mu_\infty+\nabla \ln \mu_\infty$ and $\mathbf{b}^N-\nabla \ln G_t^N$ are locally Lipschitz since the solutions to \eqref{eq:Loui} or \eqref{eq:NLPDE} have smooth and positive density for any time $t>0$. Then observe that
	\begin{align*}
		\frac{d}{dt}\int_{\R^{dN}}G_t^N\ln G_t^N&=\int_{ \R^{2N} }\pa_tG^N (1+\ln G^N ) \\
		&\quad = \frac{1}{N}\sum_{i,j=1}^N\int_{ \R^{2N} } \nabla_i \cdot \lt( 
		b(x_i,x_j) G^N \rt) (1+\ln G^N )+ \int_{ \R^{2N}}\Delta G^N(1+\ln G^N)\\
		&\quad = \frac{1}{N}\sum_{i,j=1}^N\int_{ \R^{2N} } \nabla_i \cdot \lt( 
		b(x_i,x_j)  \rt)G^N-\int_{ \R^{2N} }\lt|\nabla \ln G^N\rt|^2G^N,
	\end{align*}
	and then
	\begin{align*}
		\int_0^t\int_{ \R^{2N} }\lt|\nabla \ln G_s^N\rt|^2G_s^N \ ds \leq &\int_{\R^{dN}}G_0^N\ln G_0^N-\int_{\R^{dN}}G_t^N\ln G_t^N\\
		&-\sum_{i=1}^N\int_0^t \int_{ \R^{2N} } \Delta V(x_i) G_s^N \ ds-\frac{\e}{N}\sum_{i,j=1}^N\int_0^t \int_{ \R^{2N} } \Delta W(x_i-x_j))G_s^N \ ds\\
		&=\int_{\R^{dN}}G_0^N\ln G_0^N-\int_{\R^{dN}}G_t^N\ln G_t^N-\sum_{i=1}^N\int_0^t \int_{ \R^{2N} }12d(|x_i|^2-a/6) G_s^N \ ds-2d\e N
	\end{align*}	 
	
	Moreover
	\begin{align*}
			\int_{ \R^{2N} }\lt|\mathbf{b}^N\rt|^2G_s^N  = &\sum_{i=1}^N\int_{\R^{dN}} \lt|  (4|x_i|^2-2a)x_i -2\e \mej  (x_i-x_j) \rt|^2G_s^N \\
			&\leq  \sum_{_i=1}^{N} \int_{\R^{dN}} C(|x_i|^6+1)	G_s^N
	\end{align*}	
	Then using Lemma \ref{lem:mom}, for any $t\geq 0$ it holds
	\[
	\int_0^t\int_{ \R^{dN} }\lt|\mathbf{b}^N-\nabla \ln G_s^N\rt|^2G_s^N \ ds<\infty.
	\]
	Similar computations would also yield 
	\[
	\int_{ \R^{d} }\lt|\nabla V+\e\nabla W*\mu_\infty +\nabla \ln\mu_\infty \rt|^2\mu_\infty <\infty.
	\]

By Proposition \ref{prop:WJ} and Remark \ref{rq:WJ}, it holds that if $a,\e>0$ satisfy the assumptions of Theorem \ref{thm:main}, then any $\mu_\infty$ stationary solution to \eqref{eq:NLPDE} satisfies a $WJ(\kappa)$ inequality. In particular it is unique. Indeed if $\tilde{\mu}_\infty$ is another stationary solution, then it holds for any $t\geq 0$
\begin{align*} 
W_2^2(\mu_\infty,\tilde{\mu}_\infty)&\leq W_2^2(\mu_\infty,\tilde{\mu}_\infty)-\int_0^t \JJ(\tilde{\mu}_\infty|V,W,\mu_\infty)ds\\
&\leq W_2^2(\mu_\infty,\tilde{\mu}_\infty)-t\kappa W_2^2(\mu_\infty,\tilde{\mu}_\infty)
\end{align*}
and then $W_2^2(\mu_\infty,\tilde{\mu}_\infty)=0$. Then by by definition of $\FF_N$ it holds for any $N \geq 2$
\begin{align*}
\FF_N(\nabla V+ \e \nabla W,\mu)	&\leq \frac{1}{N^2}\sum_{i\neq j}^N\int_{\R^{dN}} \int_{\R^d} \lt| \lt(\nabla V(x_i)+\e  \nabla W(x_i-x_j-z)\rt)-\lt(\nabla V(x_i)+\e  \nabla W(x_i-x_j)\rt)\rt|^2 \mu(dz)  \mu^{\otimes N}(dx)\\
&\leq \frac{\e^2}{N^2}\sum_{i\neq j}^N\int_{\R^{dN}} \int_{\R^d} \lt|   (x_i-x_j-z)- (x_i-x_j)\rt|^2 \mu(dz)  \mu^{\otimes N}(dx)\leq \e^2 \int_{\R^d}|z|^2\mu(dz)
\end{align*}	
So that using Propositions \ref{prop:main}, \ref{prop:WJ} and Lemma \ref{lem:mom}, we obtain for $\eta<\kappa$
 \begin{align*}
 	W_2^2&(\mu_t^{\otimes N},G^N_t)\leq  W_2^2(\mu_0^{\otimes N},G^N_0)-(\kappa-\eta)\int_0^t W^2_2(\mu_s^{\otimes N},G^N_s) ds+t\frac{2}{\eta}\e^2\sqrt{(a+\e)^2+d},
 \end{align*}
 and Theorem \ref{thm:main} is proved, using Gronwall's inequality. \newline

We now turn to the proof of Corollary \ref{cor:main}. Consider $(\mu_t)_{t\geq 0}$ the solution to \eqref{eq:NLPDE} starting from $\mu_0$. Then by \cite[Theorem 23.9]{Vil1} we have
\begin{align*}
	W_2^2(\mu_t,\mu_\infty)=W_2^2(\mu_0,\mu_\infty)-\int_0^t\JJ(\mu_s|V,W,\mu_\infty)ds,
\end{align*}
and since $\mu_\infty$ satisfies a $WJ(\kappa)$ inequality we deduce
\[
W_2^2(\mu_t,\mu_\infty)\leq W_2^2(\mu_0,\mu_\infty)e^{-\kappa t}.
\]
Then using Proposition \ref{prop:main}, Lemma \ref{lem:mom} and \cite[Lemma 3.2]{BGG1} we obtain for any $\eta>0$
\begin{align*} 
W_2^2(G_t^N,\tilde{\mu}^{\otimes N}_t)&\leq W_2^2(G_0^N,\mu^{\otimes N}_0)-\int_0^t \JJ(G_s^N|V,W,\tilde{\mu}^{\otimes N}_s)ds+\eta\int_0^tW_2^2(G_s^N,\tilde{\mu}^{\otimes N}_s)+\eta^{-1} \int_0^t \FF_N(\nabla V+\e\nabla W, \mu_s)ds \\
&\leq W_2^2(G_0^N,\mu^{\otimes N}_0)+ 4(a+\e+\eta) \int_0^t W_2^2(G_s^N,\tilde{\mu}_s^{\otimes N})ds+\frac{\e^2}{2(a+\e)}\int_0^t \lt( \int_{\R^d}|z|^2\mu_s(dz) \rt)  ds,
\end{align*}
which yields the classical synchronous coupling result
\[
W_2^2(G_t^N,\mu_t^{\otimes N})\leq (W_2^2(G_0^N,\mu^{\otimes N}_0)+C)e^{4(a+\e+\eta) t}.
\]
So that we have
\begin{align*}
\frac{W_2^2(G_t^N,\mu_t^{\otimes N})}{N}\leq \begin{cases}
& \frac{W_2^2(G_0^N,\mu_\infty^{\otimes N})}{N}e^{- \alpha t}+\frac{C}{N}+W_2^2(\mu_0,\mu_\infty)e^{-\kappa t} \text{ if } t>T_N  \\ 
& \lt(\frac{W_2^2(G_0^N,\mu^{\otimes N}_0)}{N}+\frac{C}{N}\rt)e^{4(a+\e+\eta) t} \text{ if } t<T_N
\end{cases}.
\end{align*}
Choosing $T_N=\delta \ln N $
\begin{align*}
\frac{W_2^2(G_t^N,\mu_t^{\otimes N})}{N}\leq \begin{cases}
& \lt(\frac{W_2^2(G_0^N,\mu_\infty^{\otimes N})}{N}+W_2^2(\mu_0,\mu_\infty)\rt)N^{-\alpha\delta }+\frac{C}{N} \text{ if } t>T_N  \\ 
& \lt(\frac{W_2^2(G_0^N,\mu_0)}{N}+\frac{C}{N}\rt)N^{4(a+\e+\eta) \alpha} \text{ if } t<T_N
\end{cases}
\end{align*}
We now assume that $G^N_0=\mu_0^{\otimes N}$, so that
\begin{align*}
	\frac{W_2^2(G_t^N,\mu_t^{\otimes N})}{N}\leq C  \max\lt( N^{-(1\vee \alpha \delta )}  , N^{-(1-4(a+\e+\eta) \delta)} \rt)
\end{align*}
for a constant $C>0$ which depends only on $a,\e$ and $\mu_0$ and $\mu_\infty$.

%
%

\appendix
\section{Superaddittivity of the trace of the inverse}

\begin{proposition}
	\label{prop:sym} Let be $N\geq 2$ and $S^N\in \MM_{dN}(\R)$ a symmetric positive definite matrix and its block $(S^N_{i,j})_{i,j=1\cdots,N}\in \MM_d(\R)$, then it holds
	\[
	Tr[ (S^N)^{-1} ]\geq \sum_{i=1}^NTr[ (S^N_{i,i})^{-1}]
	\]
	
\end{proposition}

\begin{proof}
	\textit{Step one}: scalar case \newline
	Let us denote the property
	\[
	\PP(N): \forall S\in \MM_{N}(\R) \ \text{symmetric positive definite matrix, it holds} \ Tr\lt[S^{-1} \rt]\geq \sum_{i=1}^N \frac{1}{S_{ii}}.
	\]
	We proceed by recursion, let be $N=2$ and let be 
	\[
	A=\begin{pmatrix}
	a& c \\ 
	c& b
	\end{pmatrix} ,
	\]
	a symmetric positive definite matrix.  Note that it holds $ab-c^2>0$, $a,b>0$ and 
	\[
	A^{-1}=(ab-c^2)^{-1}\begin{pmatrix}
	b& -c \\ 
	-c & a
	\end{pmatrix} ,
	\]
	So that
	\[
	Tr[A^{-1}]=\frac{b}{ab-c^2}+\frac{a}{ab-c^2}\geq \frac{1}{a}+\frac{1}{b},
	\]
	and $\PP(2)$ holds true.\newline 
	
	Assume now that $\PP(N)$ holds true for some $N\geq 2$. Let be $\mathcal{S}^{N+1}\in M_{N+1}$ be a symmetric definite positive matrix which we write as
	\[
	\mathcal{S}^{N+1}=\begin{pmatrix}
	\mathcal{S}^{N}& \delta_N  \\ 
	\delta^{\perp}_N& z
	\end{pmatrix} .
	\]
	Note that necessarily $\mathcal{S}^{N}\in \MM_{N}(\R)$ is symmetric positive definite (and so is $(\mathcal{S}^{N})^{-1}$), $\delta_N\in \R^N$ and $z>\delta_N^{\perp}(\mathcal{S}^{N})^{-1}\delta_N$. Then
	\[
	(\mathcal{S}^{N+1})^{-1}=\begin{pmatrix}
	(\mathcal{S}^{N})^{-1}+\frac{(\mathcal{S}^{N})^{-1}\delta_N\delta_N^{\perp}(\mathcal{S}^{N})^{-1}}{z-\delta_N^{\perp}(\mathcal{S}^{N})^{-1}\delta_N}& -\frac{(\mathcal{S}^{N})^{-1}\delta_N}{z-\delta_N^{\perp}(\mathcal{S}^{N})^{-1}\delta_N}  \\ 
	-\frac{\delta_N^{\perp}(\mathcal{S}^{N})^{-1}}{z-\delta_N^{\perp}(\mathcal{S}^{N})^{-1}\delta_N}& \frac{1}{z-\delta_N^{\perp}(\mathcal{S}^{N})^{-1}\delta_N}
	\end{pmatrix} 
	\]
	so that
	\begin{align*}
	Tr[(\mathcal{S}^{N+1})^{-1}]&=Tr\lt[ (\mathcal{S}^{N})^{-1}+\frac{(\mathcal{S}^{N})^{-1}\delta_N\delta_N^{\perp}(\mathcal{S}^{N})^{-1}}{z-\delta_N^{\perp}(\mathcal{S}^{N})^{-1}\delta_N} \rt] +\frac{1}{z-\delta_N^{\perp}(\mathcal{S}^{N})^{-1}\delta_N}\\
	& =Tr\lt[ (\mathcal{S}^{N})^{-1}\rt]+\frac{1}{z-\delta_N^{\perp}(\mathcal{S}^{N})^{-1}\delta_N} Tr\lt[(\mathcal{S}^{N})^{-1}\delta_N\delta_N^{\perp}(\mathcal{S}^{N})^{-1}\rt] +\frac{1}{z-\delta_N^{\perp}(\mathcal{S}^{N})^{-1}\delta_N}.
	\end{align*}
	Also note that
	\begin{align*}
	Tr\lt[(\mathcal{S}^{N})^{-1}\delta_N\delta_N^{\perp}(\mathcal{S}^{N})^{-1}\rt]=\lt|(\mathcal{S}^{N})^{-1}\delta_N \rt|^2\geq 0.
	\end{align*}
	Finally, we have
	\begin{align*}
	Tr[(\mathcal{S}^{N+1})^{-1}]&\geq  Tr\lt[ (\mathcal{S}^{N})^{-1}\rt]+\frac{1}{z-\delta_N^{\perp}(\mathcal{S}^{N})^{-1}\delta_N}\\
	&\geq Tr\lt[ (\mathcal{S}^{N})^{-1}\rt]+\frac{1}{z}
	\end{align*}
	which concludes the first step.\newline 
	
	\textit{Step two} multidimensional case.
	
	Let us denote the property
	\[
	\PP(N): \forall S\in \MM_{dN}(\R) \ \text{symmetric positive definite matrix, it holds} \  Tr\lt[S^{-1} \rt]\geq \sum_{i=1}^N Tr[(S_{ii})^{-1}].
	\]
	We proceed by recursion, let be $N=2$ and let be $S\in \MM_{2d}(\R)$
	\[
	S=\begin{pmatrix}
	A& C \\ 
	C^{\perp}& B
	\end{pmatrix} ,
	\]
	a symmetric positive definite matrix. Then 
	\[
	S^{-1}=\begin{pmatrix}
	A& C \\ 
	C^{\perp}& B
	\end{pmatrix}^{-1} = \begin{pmatrix}
	A^{-1}+A^{-1}C(B-C^\perp A^{-1}C)^{-1}C^\perp A^{-1}& -A^{-1}C(B-C^\perp A^{-1}C)^{-1} \\ 
	-(B-C^\perp A^{-1}C)^{-1}CA^{-1}& (B-C^\perp A^{-1}C)^{-1}
	\end{pmatrix} ,
	\]
	So that
	\begin{align*}
	Tr[S^{-1}]&=Tr\lt[ A^{-1}\rt]+Tr\lt[ (B-C^\perp A^{-1}C)^{-1}\rt]+Tr\lt[A^{-1}C(B-C^\perp A^{-1}C)^{-1}C^\perp A^{-1} \rt]\\
	&\geq Tr\lt[ A^{-1}\rt]+Tr\lt[ (B-C^\perp A^{-1}C)^{-1}\rt],
	\end{align*}
	Let be $O\in \MM_d(\R)$ be an orthonormal matrix such that
	\[
	OBO^{-1}=\begin{pmatrix}
	\lambda_1& 0 & \cdots &  & 0 \\ 
	0& \lambda_2  &  &  & 0 \\ 
	&  &  &  &  \\ 
	&  &  &  & 0 \\ 
	0&  & \cdots &  & \lambda_d 
	\end{pmatrix} 
	\]
	So that 
	\begin{align*}
	O(B-C^\perp A^{-1}C)O^{-1}=\begin{pmatrix}
	\lambda_1-a_{1,1}& a_{1,2} & \cdots &  & a_{1,d} \\ 
	& \lambda_2-a_{2,2}  &  &  &  \\ 
	&  &  &  &  \\ 
	&  &  &  &  \\ 
	a_{1,d}	&  &  &  & \lambda_d- a_{d,d}
	\end{pmatrix} 
	\end{align*}	
	with 
	\[
	(a_{i,j})_{i,j=1\cdots,N}=	OC^\perp A^{-1}CO^{-1}
	\]
	Observe that since $O(B-C^\perp A^{-1}C)O^{-1}$ is symmetric definite positive, it holds $\lambda_k>a_{k,k}>0$ for any $k=1\cdots,d$. So that 
	\begin{align*}
	Tr\lt[ (B-C^\perp A^{-1}C)^{-1} \rt]&=Tr\lt[ O^{-1}(B-C^\perp A^{-1}C)^{-1}O \rt]\\
	&\geq \sum_{k=1}^d \frac{1}{\lambda_k-a_{k,k}}\\
	&\geq \sum_{k=1}^d \frac{1}{\lambda_k}=Tr\lt[(B^{-1})\rt],
	\end{align*}
	where we used the commutativity of the trace in the first and last line, and the first step of this proof in the second. And $\PP(2)$ holds true.\newline 
	
	Assume now that $\PP(N)$ holds true for some $N\geq 2$. Let be $\mathcal{S}^{N+1}\in M_{d(N+1)}$ be a symmetric matrix which we write as
	\[
	\mathcal{S}^{N+1}=\begin{pmatrix}
	\mathcal{S}^{N}& \delta_N  \\ 
	\delta^{\perp}_N& Z
	\end{pmatrix} 
	\]
	with $\mathcal{S}^{N}\in \MM_{dN}(\R)$ symmetric positive definite, $\delta_N\in \MM_{d,dN}(\R)$ and $Z\in \MM_d(\R)$. Then denoting $\tilde{Z}=(Z-\delta_N^{\perp}(\mathcal{S}^{N})^{-1}\delta_N)^{-1}$ 
	\[
	(\mathcal{S}^{N+1})^{-1}=\begin{pmatrix}
	(\mathcal{S}^{N})^{-1}+\tilde{Z}(\mathcal{S}^{N})^{-1}\delta_N\delta_N^{\perp}(\mathcal{S}^{N})^{-1}& -\tilde{Z}(\mathcal{S}^{N})^{-1}\delta_N  \\ 
	-\tilde{Z}\delta_N^{\perp}(\mathcal{S}^{N})^{-1}& \tilde{Z}
	\end{pmatrix} 
	\]
	so that
	\begin{align*}
	Tr[(\mathcal{S}^{N+1})^{-1}]&=Tr\lt[ (\mathcal{S}^{N})^{-1}\rt]+Tr\lt[(\mathcal{S}^{N})^{-1}\delta_N\tilde{Z}\delta_N^{\perp}(\mathcal{S}^{N})^{-1} \rt]+Tr[\tilde{Z}]\\
	\end{align*}
	But with the same argument as the one used above we have
	\begin{align*}
	Tr\lt[(\mathcal{S}^{N})^{-1}\delta_N\tilde{Z}\delta_N^{\perp}(\mathcal{S}^{N})^{-1} \rt]\geq 0,	\ Tr[\tilde{Z}]\geq Tr\lt[Z^{-1}\rt],
	\end{align*}
	and the conclusion follows.
\end{proof}

%
%

\section*{Acknowledgements}
The author was supported by the Fondation des Sciences Math\'ematiques de Paris and Paris Sciences \& Lettres Universit\'e, and warmly thanks Maxime Hauray and Arnaud Guillin for many advices, comments and discussions which have made this work possible.

%
%
%
%

\end{document}